\documentclass{article}
\usepackage{fullpage}
\usepackage[utf8]{inputenc}
\usepackage{amsmath,amsfonts,amssymb,amsthm,mathtools}
\usepackage{color}
\usepackage{hyperref}

\newcommand{\Uq}{\mathcal{U}_q(\mathfrak{g})}
\newcommand{\DeltaL}{\Delta^{(L)}}

\newtheorem{theorem}{Theorem}[section]
\newtheorem{proposition}[theorem]{Proposition}
\newtheorem{lemma}[theorem]{Lemma}
\newtheorem{corollary}[theorem]{Corollary}

\newtheorem{definition}[theorem]{Definition}

\title{Stochastic duality of ASEP with two particle types via symmetry of quantum groups of rank two}

\author{Jeffrey Kuan }

\begin{document}

\maketitle

\abstract{We study two generalizations of the asymmetric simple exclusion process with two types of particles. Particles of type $1$ can jump over particles of type $2$, while particles of type $2$ can only influence the jump rates of particles of type $1$. We prove that the processes are self--dual and explicitly write the duality function. As an application, an expression for the $r$--th moment of the exponentiated current is written in terms of $r$--particle evolution.

The construction and proofs of duality are accomplished using symmetry of the quantum groups $\mathcal{U}_q(\mathfrak{gl}_3)$ and $\mathcal{U}_q(\mathfrak{sp}_4)$, with each node in the Dynkin diagram corresponding to a particle type, and the number of edges corresponding to the jump rates. 
}

\section{Introduction}
The asymmetric simple exclusion process (ASEP) is a widely studied model in mathematics and physics. Particles occupy a one--dimensional lattice, with at most one particle at each site. The particles jump to neighboring sites asymmetrically, meaning that particles will drift to either the right of the left. If the particle jumps to an occupied site, the jump is blocked. 

In \cite{BS0},\cite{CS} and \cite{TW}, there is additionally a second--class particle. This particle jumps according to the same rule as ASEP. However, if a first--class particle attempts to jump to a site occupied by a second--class particle, the particles switch positions. If a second--class particle attempts to jump to a site occupied by a first--class particle, the jump is blocked. Observe that the second--class particles do not affect the first--class particles, or in other words, the projection to the first--class particles is Markov.

This paper will introduce so-called ``semi--second--class'' particles. These are particles which can not jump over the first--class particles; however, their presence can influence the jump rates of adjacent first--class particles. Thus, the projection to first--class--particles is \textit{not} Markov. These particles will also be called ``type 1'' and ``type 2'' particles. 

Two particular sets of values for the jump rates will be studied in detail. In these two cases, we will show that the processes are self--dual and explicitly write the duality function. The duality is proved using symmetry of the rank two quantum groups $\mathcal{U}_q(\mathfrak{gl}_3)$ and $\mathcal{U}_q(\mathfrak{sp}_4)$. The use of algebra symmetries to prove duality has a well--established history (e.g. \cite{SS},\cite{S},\cite{BS1},\cite{IS}). The proofs in this work follow the method laid out in \cite{CGRS}. Recent work \cite{BCS},\cite{CP} has also developed proofs for duality using more ``direct'' (that is, without algebra) methods. It has also been previously known that ASEP with second--class particles is integrable (e.g. \cite{AB}) and satisfies $\mathcal{U}_q(\mathfrak{gl}_3)$ symmetry \cite{AR}, but the explicit representations had not been constructed. Similar models have also been shown to be integrable (e.g. \cite{C},\cite{DE}).

The remainder of this paper is organized as follows: section \ref{Overview} gives an explicit description of the processes, states the duality results as well as an application, and state the quantum group symmetry. Section \ref{Central Element} reviews the background on quantum groups and constructs a central element necessary for the proof. Section \ref{C2} finishes the proofs for the $\mathcal{U}_q(\mathfrak{sp}_4)$ case, and section \ref{A2} finishes the proofs for the $\mathcal{U}_q(\mathfrak{gl}_3)$ case. 

During the writing of this paper, another paper \cite{BS} was posted to arXiv with similar results. That paper studies the process arising from $\mathcal{U}_q(\mathfrak{gl}_3)$ symmetry and finds a duality function similar to the one presented here. The approach is different in that it uses the Perk--Schultz quantum spin chain \cite{PS} to construct the representations, rather than explicitly constructing a central element. That paper also explicitly constructs all invariant measures and proves an interesting sum rule for the duality functions, neither of which are addressed here.

\textbf{Acknowledgments}. The author would like to thank Alexei Borodin and Ivan Corwin for helpful conversations. The author was partially supported by a National Science Foundation Graduate Research Fellowship.

\section{Overview}\label{Overview}

\subsection{Description}
Consider a one--dimensional lattice. Each lattice site has three possible states: either empty, occupied by a first--class particle, or occupied by a semi--second--class particle. Describe a particle configuration by $ \xi= \{\xi_i\}$ where $\xi_i\in \{0,1,2\}$ for each lattice site $i$, corresponding to an empty site, occupation by a first-class particle, and occupation by a semi--second--class particle, respectively. 

Each particle has two independent exponential clocks, one for left jumps and one for right jumps. The rate of the left clock depends on the state of the site to the left of the particle, and the rate of the right clock depends on the state of the site to the right of the particle. Let $L(i,j)$ denote the rate of the left clock of the $i$--th class particle when the site to the left has state $j$, and similarly denote $R(i,j)$. 

The particles interact according to the following rules: If a first--class particle attempts to jump to a site occupied by a semi--second--class particle, then the particles switch positions. If a semi--second--class particle attempts to jump to a site occupied by a first--class particle, the jump is blocked.  This implies that 
$$
L(2,1)=R(2,1)=0.
$$
If a first--class particle attempts to jump to a site occupied by another first--class particle, then the jump is blocked. The same holds for the semi--second--class particle. This means that
$$
L(1,1)=L(2,2)=R(2,2)=R(1,1)=0.
$$

This leaves six remaining jump rates, $L(1,0),R(1,0),L(2,0),R(2,0),L(1,2),R(1,2)$. Observe that if
$$
L(1,0)=L(1,2) , \quad R(1,0)=R(1,2), 
$$
then the first class particles evolve independently of the semi--second--class particles. In other words, the behavior of the first--class particles is Markov. In this case, the semi--second--class particles have been described in the literature as second--class particles (see e.g. \cite{S,TW}). In general, however, the semi--second--class particles still affect the jump rates of the first--class particles, even if they can not jump over them. Also observe that if the six jump rates are all multiplied by the same positive constant, then this corresponds to rescaling the time, and hence has no effect on the interaction of the particles. 

In this paper, we will consider two particular sets of values for the jump rates. The first is when
$$
L(1,0)=L(2,0)=L(1,2)=1, \quad R(1,0)=R(2,0)=R(1,2)=q^{-2}.
$$
This is called $\textit{spin } 1/2 \textit{ type } A_2 \textit{ ASEP}$, or ASEP with second--class particles. Here, the asymmetry parameter is $q$ for all particles. The second set of values for the jump rates is when
$$
L(1,0)=L(2,0)=1, \quad L(1,2)=a, \quad R(1,0)=R(2,0)=q^{-2}, \quad R(1,2)=aq^{-4},
$$
where $a$ solves $(q^{-4}+q^6)a=q^2(q^2+q^{-2})^2$. This is called $\textit{spin } 1/2 \textit{ type } C_2 \textit{ ASEP }$. In other words, the asymmetry parameter is $q$ for particles of type $1$ and $2$, and $q^2$ when particles of type $1$ and type $2$ interact. The reasons for the names will become clear later in the paper.

\subsection{Duality results}
Let us review the definition of duality. 

\begin{definition}
Suppose that  $X(t)$ and $Y(t)$ are Markov processes on state spaces $X$ and $Y$ respectively. Given a function $D$ on $X\times Y$, let $\mathcal{S}_D\subseteq X\times Y$ be the set of all $(x,y)$ such that
$$
\mathbb{E}_x[D(x(t),y)] = \mathbb{E}_y[D(x,y(t))],
$$
where on the left hand side, the process $x(t)$ starts at $x(0)=x$, and on the right hand side, the process $y(t)$ starts at $y(0)=y$. If $\mathcal{S}_D = X\times Y$, then we say that $X(t)$ and $Y(t)$ are dual with respect to $D(x,y)$. If furthermore, $X(t)$ and $Y(t)$ are the same process, then we say that $X(t)$ is self--dual.
\end{definition}

In order to write the explicit formula for the duality functions, define 
\begin{align*}
N^L_i(\eta) &= \sum_{j=1}^{i-1} 1_{\{\eta_j\neq 0\}}, \quad \quad \tilde{N}^L_i(\eta) = \sum_{j=1}^{i-1} 1_{\{\eta_j = 1\}}, \\
N^R_i(\eta) &= \sum_{j=i+1}^{L} 1_{\{\eta_j\neq 0\}}, \quad \quad \tilde{N}^R_i(\eta) = \sum_{j=i+1}^{L} 1_{\{\eta_j = 1\}}. \\
\end{align*}

\begin{theorem}\label{SDF} In the $A_2$ case, if $\xi$ is the particle configuration with particles of type $1$ at $n_1,\ldots,n_r$ and particles of type $2$ at $m_1,\ldots,m_{r'}$, then the function $D(\cdot,\cdot)$ defined by 
$$
D(\eta,\xi) = \prod_{s=1}^r 1_{\{\eta_{n_s}=1\}}q^{2\tilde{N}^R_{n_s}(\eta)+2n_s} \prod_{s'=1}^{r'} 1_{\{\eta_{m_{s'}}\neq 0\}}q^{2N^R_{m_{s'}}(\eta)+2m_{s'}} 
$$
is a self--duality function.
\end{theorem}
This function is similar to Proposition 2 of \cite{IS} or (3.12) of \cite{S}. Indeed, if $\xi$ only contains type $2$ particles, one recovers the self--duality function for the projection of type $A_2$ ASEP to the number of particles, which is still ASEP. If $\xi$ only contains type $1$ particles, one recovers the self--duality function for the projection of type $A_2$ ASEP to the type $1$ particles, which is again still ASEP.

%This duality function allows for a computation of the $r$th moment of $\tilde{Q}_i(\eta)=1_{\{\eta_i=2\}}q^{2\tilde{N}_i^R(\eta)}$ at $r$ different location sites. Observe that this can be written as 
%$$
%\frac{q^{2\tilde{N}_{i-1}^R(\eta)} - q^{2\tilde{N}_{i}^R(\eta)} }{q^2-1}
%$$
%and let $Q_i(\eta) = q^{2\tilde{N}_i^R(\eta)}$.  Write $Q_i(t),\tilde{Q}_i(t)$ for $Q_i(\eta(t)),\tilde{Q}_i(\eta(t))$.Then we have a proposition analogous to Proposition 3 of \cite{IS}:

%\begin{proposition}\label{Weyl}
%\begin{align*}
%\mathbb{E}[Q_x^r(t)] &= \sum_{k=0}^r (q^{2r}-1)\cdots (q^{2(r-k+1)}-1)\\
%& \quad \times \sum_{\xi} \sum_{y} \mathbb{E}[\tilde{Q}_{y_1}(0)\cdots \tilde{Q}_{y_k}(0)] \mathbb{P}(y \rightarrow \xi \text{ in time } t)
%\end{align*}
%where the summation is over all $\xi$ with right--most--particle to the left of position $x$, and all %configurations $y$ with $k$ particles of type $2$. 
%\end{proposition}

In the $C_2$ case, we give two duality functions:
\begin{theorem}\label{C2Duality} 
(1) In the $C_2$ case, the function 
$$
D(\eta,\xi) = \prod_{i=1}^L \left( 1_{\{\xi_i = \eta_i=1\}}q^{2(i-1)} +  1_{\{\xi_i = \eta_i=2\}}q^{2(i-1 + N^L_i(\eta) + N^L_i(\xi))}  +  1_{\{\xi_i=1,\eta_i=2\}}q^{2(N_i^L(\eta) + i-1+ 2N_i^L(\xi)-\tilde{N}_i^L(\xi) )}\right)
$$
is a self--duality function.

(2) In the $C_2$ case, there is a function $D$ such that
$$
\mathcal{S}_D = \{0,1,2\}^L \times \{0,1\}^L.
$$ 
Explicitly, if for $n_1<\ldots<n_r,$  the particle configuration $\xi^{(n_1,\ldots,n_r)}$ is defined by 
$$
\xi^{(n_1,\ldots,n_r)}_i 
= \begin{cases}
1, i \in \{n_1,\ldots,n_r\} \\
0, i \notin \{n_1,\ldots,n_r\} 
\end{cases}
$$
then
$$
D(\eta,\xi^{(n_1,\ldots,n_r)}) = \prod_{s=1}^r 1_{\{\eta_{n_s}\neq 0\}} q^{ 2N_{n_s}^R(\eta)+2n_s}
$$
\end{theorem}

\noindent \textbf{Remark}. Theorem \ref{C2Duality}(2) can also be stated as ``spin $1/2$ type $C_2$ ASEP is dual to usual ASEP with respect to $D$.'' Also observe that the function $D(\cdot,\cdot)$ only detects the number, not the type, of particles. The projection of type $A_2$ and $C_2$ ASEP to particle occupation is simply the usual ASEP, and the duality function matches that from \cite{S}. The interest lies in that $D(\cdot,\cdot)$ can be constructed from the representation theory of $\mathcal{U}_q(\mathfrak{sp}_4)$, which will be seen below.
\subsection{Construction}\label{Construction}

In \cite{CGRS}, there is a general description of how to construct particle systems from quantum groups, as well as finding self--duality functions for these particle systems.  Let us review the idea in several steps.

The first step is to consider the quantum group $\mathcal{U}_q(\mathfrak{g})$ for some finite--dimensional simple Lie algebra $\mathfrak{g}$. Find an explicit central element $C\in \mathcal{U}_q(\mathfrak{n}_-)\mathcal{U}_q(\mathfrak{h}) \mathcal{U}_q(\mathfrak{n}_+)$.

Next, consider a finite--dimensional irreducible representation  $V$ of $\mathcal{U}_q(\mathfrak{g})$ with a basis $v_1,\ldots,v_d$ consisting of weight space vectors. If $v_1$ denotes the highest weight vector and $v_d$ denotes the lowest weight vector, then compute the value of $a$ for which $\Delta(C)(v_1\otimes v_1)=a v_1\otimes v_1$. Now compute the $d\times d$ matrix of $A:=\Delta(C-a)$ acting on $V\otimes V$ with respect to the basis $\{v_i\otimes v_j, 1\leq i,j\leq d\}$. Observe that since $C-a$ is still central. Assume that this matrix has non--positive diagonal entries, and non--negative off--diagonal entries (this will not always be true). 

Now consider the operator on $V^{\otimes L}$ defined by 
$$
\sum_{i=1}^{L-1} 1^{\otimes i-1} \otimes A \otimes 1^{\otimes L-i-1}.
$$
Suppose that we have a vector $g \in V^{\otimes L}$ such that $A^{(L)}g=0$. It is possible to find such a vector by applying elements of $\mathcal{U}_q(\mathfrak{n}_+)$ to the lowest weight vector $v_d\otimes \cdots \otimes v_d$ (in physics language, this is applying creation operators to the vacuum state).  Write $g$ in terms of the canonical basis
$$
\sum_{1\leq i_1,\ldots,i_L \leq d} g(i_1,\ldots,i_L) v_{i_1} \otimes \cdots \otimes v_{i_L}
$$
Assume that $g(i_1,\ldots,i_L)$ is always positive and define $G$ to the diagonal operator on $V^{\otimes L}$ defined by 
$$
G(v_{i_1} \otimes \cdots \otimes v_{i_L}) =  g(i_1,\ldots,i_L) v_{i_1} \otimes \cdots \otimes v_{i_L}.
$$
By the assumptions on $A$, the matrix $\mathcal{L} = G^{-1}A^{(L)}G$ is the generator of a continuous--time Markov chain on the state space $\{1,\ldots,d\}^L$.

Finally, if $A^{(L)}$ is self--adjoint on $V^{\otimes L}$ and $S$ is an operator that commutes with $A$, then $D=G^{-1}SG^{-1}$ is a self--duality function for the particle system generated by $\mathcal{L}$. 

This paper will consider the situation in which $\mathfrak{g} = \mathfrak{sp}_4$ or $\mathfrak{gl}_3$ and $V$ is the fundamental representation. The precise statements are as follows:

\begin{theorem}
There exists a central element $C\in \mathcal{U}_q(\mathfrak{gl}_3)$ and an operator $G$ on $V^{\otimes L}$ such that for 
$$
A^{(L)} := \sum_{i=1}^{L-1} 1^{\otimes i-1} \otimes \Delta(C) \otimes 1^{\otimes L-i-1},
$$
$G^{-1}A^{(L)}G$ is the generator of spin $1/2$ type $A_2$ ASEP on the lattice $\{1,\ldots,L\}$ with domain wall boundary conditions. 

The self--duality function $D(\cdot,\cdot)$ in Theorem \ref{SDF} is of the form $G^{-1}SG^{-1}$ for a symmetry $S$ of $A^{(L)}$.
\end{theorem}

In the following theorem, the notation for $A^{(L)}$ and $\tilde{A}^{(L)}$ are the same.

\begin{theorem}\label{C2Construction}
There exists a central element $C\in \mathcal{U}_q(\mathfrak{sp}_4)$ and operators $G_{\epsilon}$ on $V^{\otimes L}$ such that the limit $\lim_{\epsilon\rightarrow 0} G_{\epsilon}^{-1}A^{(L)}G_{\epsilon}$ is the generator of spin $1/2$ type $C_2$ ASEP on the lattice $\{1,\ldots,L\}$ with domain wall boundary conditions.

The functions in Theorem \ref{C2Duality} are of the form $G^{-1}SG^{-1}$ for a symmetry $S$ of $A^{(L)}$.
\end{theorem}

\section{Central Element}\label{Central Element}
The first step is to find a suitable central element in $\mathcal{U}_q(\mathfrak{g})$. This will be done with the quantum Harish--Chandra isomorphism. In principle, one could directly check that the resulting element is central using only the commutation relations, but the whole proof is presented here in order to make the construction less mysterious and more applicable for other Lie algebras.

Given a simple Lie algebra $\mathfrak{g}$ of rank $n$, the quantum group $\Uq$ is the Hopf algebra generated by $\{e_i,f_i,k_i\},1\leq i\leq n$ with relations
$$
[e_i,f_j] = \delta_{ij}\frac{k_i-k_i^{-1}}{q_i-q_i^{-1}} , \quad [k_i,k_j]=0
$$
$$
[e_i,e_j]=[f_i,f_j]=0,\ \  i\neq j+1,
$$
$$
k_{i}e_{j}= q^{(\alpha_i,\alpha_j)}e_{j}k_{i} \quad k_{i}f_{j}= q^{-(\alpha_i,\alpha_j)}f_{j}k_{i} 
$$
$$
\sum_{r=0}^{-a_{ij}} \binom{-a_{ij}}{r}_q e_i^{r}e_{j}e_i^{-a_{ij}-r}=0
$$
where
$$
q_i = q^{(\alpha_i,\alpha_i)/2}
$$
$$
\binom{n}{m}_q = \frac{(n)_q!}{(m)_q!(n-m)_q!}, \quad (n)_q! = \prod_{k=1}^n (k)_q, \quad (n)_q = \frac{q^n-q^{-n}}{q-q^{-1}}
$$
and
$$
a_{ij} = \frac{2(\alpha_i,\alpha_j)}{(\alpha_i,\alpha_i)}
$$
is the Cartan matrix. (Recall that $(\alpha_i,\alpha_i)=2$ for short roots and $4$ for long roots.) The co--product is
$$
\Delta(e_i) = e_i \otimes 1  + k_i \otimes e_i \quad \Delta(f_i) = 1\otimes f_i + f_i\otimes k_i^{-1}
$$
and the antipode is 
$$
S(e_i) = -k_i^{-1}e_i \quad S(f_i) = -f_i k_i, \quad S(k_i) = k_i^{-1},
$$
We will also use Greek letter subscripts $k_{\alpha}$ to denote the $k_i$, when it is notationally more convenient to do so, and $k_{\alpha+\beta}$ denotes $k_{\alpha}k_{\beta}$.

Letting $\mathfrak{b}_{\pm}$ denote the Borel subalgebras, there is a pairing (see Proposition 6.12 of \cite{J}) on $ \mathcal{U}_q(\mathfrak{b}_-) \times \mathcal{U}_q(\mathfrak{b}_+)$ defined by 
$$
\langle k_{\alpha}, k_{\beta}\rangle = q^{-(\alpha,\beta)_{\mathfrak{g}}}, \quad \langle f_i,e_j\rangle = \frac{-\delta_{ij}}{q_i-q_i^{-1}}, \quad \langle k_i,e_j\rangle = \langle f_i,k_j\rangle = \langle 1,e_i, \rangle=\langle f_j,1\rangle=0, \quad \langle 1, 1\rangle=1
$$
$$
\langle y, x\cdot x' \rangle = \langle \Delta(y), x'\otimes x\rangle, \quad \langle y \cdot y', x\rangle = \langle y \otimes y', \Delta(x)\rangle
$$
where $( \cdot,\cdot ){\mathfrak{g}}$ is an invariant, non--degenerate invariant symmetric bilinear form on $\mathfrak{h}^*$. Furthermore, according to Lemma 6.16 of \cite{J},
\begin{equation}\label{anti}
(\omega(x),\omega(y))=(y,x)=(\tau(y),\tau(x))
\end{equation}
where $\omega$ is the automorphism and $\tau$ is the antiautomorphism defined by
\begin{align*}
\omega(e_i)=f_i, \quad \omega(f_i)=e_i, \quad \omega(k_i)=k_i^{-1}\\
\tau(e_i)=e_i, \quad \tau(f_i)=f_i, \quad \tau(k_i)=k_i^{-1}
\end{align*}

Let $V$ be the fundamental representation of $\mathfrak{g}$ and let $\{v_{\mu}\}$ be a basis of $V$ such that $v_{\mu} \in V[\mu]$, the $\mu$--weight space of $V$.  For any $\mu \geq \lambda$ such that $V[\mu]$ and $V[\lambda]$ are nonzero, let $e_{\mu\lambda}$ and $f_{\lambda\mu}$ be elements of $U^+$ and $U^-$ respectively such that $e_{\mu\lambda}v_{\lambda}=v_{\mu}$ and $f_{\lambda\mu}v_{\mu}=v_{\lambda}$, Let $\rho$ be half the sum of the positive roots of $\mathfrak{g}$, and recall that $(2\rho,\alpha)=(\alpha,\alpha)$ for the simple roots $\alpha$. 

\begin{lemma}\label{CentralLemma}
If $q$ is not a root of unity and $2\mu$ is in the root lattice of $\mathfrak{g}$ for all weights $\mu$ of $V$, then the element
\begin{equation}\label{central}
\sum_{\mu \geq\lambda }  q^{(\mu-\lambda,\mu)}  q^{-(2\rho,\mu)} e_{\mu\lambda}^*  k_{-\lambda-\mu} f_{\lambda\mu}^*
\end{equation}
is central in $\Uq$, where the star $^*$ denotes the dual element under $\langle,\cdot,\cdot\rangle$. 
\end{lemma}
\begin{proof}
By following the construction of the Harish--Chandra isomorphism in \cite{J}, one sees that
$$
\sum_{\mu \geq\lambda }  q^{(\nu,\mu)} q^{-(2\rho,\lambda)} q^{-(2\rho,\nu)} e_{\mu\lambda}^* k_{\nu}k_{-2\lambda-2\nu} f_{\lambda\mu}^*, \quad \nu:=\mu-\lambda
$$
is central, which simplifies to \eqref{central}.
\end{proof}

\subsection{$\mathfrak{sp}_4$} 
Recall that $\mathfrak{sp}_{2n}$ is the rank $n$  Lie algebra consisting of $2n\times 2n$ matrix of the form
$$
\left\{
\left(
\begin{tabular}{cc}
$A $& $B $\\
$C$ & $D$
\end{tabular}
\right)
: A=-D^T, B=B^T, C=C^T
\right\}
$$
Letting $E_{ij}$ denote the matrix with a $1$ at the $(i,j)$-entry and zeroes elsewhere, define 
\begin{align*}
e_i &= E_{i,i+1} - E_{n+i+1,n+i}, \quad  f_i = E_{i+1,i} - E_{n+i,n+i+1} \\
h_i &= E_{ii} - E_{i+1,i+1} -  E_{n+i,n+i} + E_{n+i+1,n+i+1}  \\
e_n &= E_{n,2n}, \quad  f_n = E_{2n,n} \quad  h_n = E_{n,n} - E_{2n,2n} 
\end{align*}
The simple roots and fundamental weights are
\begin{align*}
\alpha_i &= \epsilon_i - \epsilon_{i+1} , 1\leq i\leq n-1\\
\alpha_n &= 2\epsilon_n, \\
\omega_i &= \epsilon_1 + \epsilon_2 + \cdots + \epsilon_n, 1\leq i\leq n
\end{align*}
where $\epsilon_i(M)=M_{ii}$.  We have that $(\alpha_i, \alpha_i) =2$ for $1\leq i\leq n-1$ and $(\alpha_n,\alpha_n)=4$. We have that $(\alpha_{n-1},\alpha_n) = -2$. When $n=2$, the Cartan matrix of $\mathfrak{sp}_4$ is simply
$$
\left(
\begin{tabular}{cc}
2 & $-2$ \\
$-1$ & 2 
\end{tabular}
\right)
$$
The Dynkin diagram of $\mathfrak{sp}_4$ is of type $C_2$, hence the notation. To make notation clearer, $k_{(1,-1)}$ denotes $k_1$ and $k_{(0,2)}$ denotes $k_2$.

Let $V$ be the fundamental representation of $\mathfrak{sp}_4$. It has a basis $v_1,v_2,v_4,v_3$ which are in the weight spaces $\epsilon_1,\epsilon_2, -\epsilon_2, -\epsilon_1$. It is immediate that the condition of lemma \ref{CentralLemma} holds. Index these and order them by $\mathbf{1} \geq \mathbf{2} \geq \mathbf{\bar{2}} \geq \mathbf{\bar{1}} $. We have that
$$
\mathbf{1} \stackrel{f_1}{\longrightarrow} \mathbf{2} \stackrel{f_2}{\longrightarrow} \mathbf{\bar{2}} \stackrel{f_1}{\longrightarrow} \mathbf{\bar{1}} .
$$ 
Here, the sum of the positive roots is
$$
2\rho = 4\epsilon_1 + 2\epsilon_2 
$$
So that
\begin{equation}\label{rho}
(-2\rho,\mu)
= 
\begin{cases}
-4, &\mu = \mathbf{1} \\
-2, &\mu = \mathbf{2} \\
2, &\mu = \mathbf{\bar{2}} \\
4, &\mu = \mathbf{\bar{1}} \\
\end{cases}
\end{equation}
In order to write the central element, the dual elements need to be calculated:

\begin{lemma} The dual elements are:
\begin{align*}
(e_1)^* &= -(q-q^{-1})f_1,  \quad  &(f_1)^* = -(q-q^{-1})e_1 \\
(e_2)^* &= -(q^2-q^{-2})f_2,  \quad & (f_2)^* = -(q^2-q^{-2})e_2 \\
(e_1e_2)^* &= (q-q^{-1})  (q^{2}f_1f_2 - f_2f_1), \quad & (f_1f_2)^* = (q-q^{-1})  (q^{2}e_1e_2 - e_2e_1) \\
(e_2e_1)^* &= (q-q^{-1})  (q^{2}f_2f_1 - f_1f_2),  \quad & (f_2f_1)^* = (q-q^{-1})  (q^{2}e_2e_1 - e_1e_2) \\
(e_1e_2e_1)^* &= (q-q^{-1})(q f_1f_1f_2 - (q^{-1} + q^3)f_1f_2f_1 + qf_2f_1f_1) \\
(e_2e_1e_1)^* &= \frac{q-q^{-1}}{q+q^{-1}} \left( f_1(q^2 f_2f_1 - f_1f_2) - (q^2 f_2f_1 - f_1f_2)f_1\right)\\
(e_1e_1e_2)^* &= \frac{q-q^{-1}}{q+q^{-1}}  (q^2 f_1f_2 - f_2f_1)f_1  - f_1(q^2 f_1f_2 - f_2f_1) \\
(f_1f_2f_1)^* &= (q-q^{-1})(q e_1e_1e_2 - (q^{-1} + q^3)e_1e_2e_1 + qe_2e_1e_1) \\
(f_2f_1f_1)^* &= \frac{q-q^{-1}}{q+q^{-1}} \left( e_1(q^2 e_2e_1 - e_1e_2) - (q^2 e_2e_1 - e_1e_2)e_1\right)\\
(f_1f_1f_2)^* &= \frac{q-q^{-1}}{q+q^{-1}} \left( q^2 e_1e_2 - e_2e_1)e_1  - e_1(q^2 e_1e_2 - e_2e_1 \right)
\end{align*}
\end{lemma}
\begin{proof}
The first two lines follow immediately from the definition of the pairing.
For the next two lines, we have that
\begin{align*}
\langle f_1f_2, e_1e_2\rangle &= \langle f_2 \otimes f_1k_2^{-1}, e_2\otimes e_1 \rangle\\
&=\langle f_2,e_2\rangle \langle f_1 \otimes k_2^{-1}, e_1 \otimes 1\rangle \\
&= (q^2-q^{-2})^{-1}(q - q^{-1})^{-1}
\end{align*}
By \eqref{anti}, 
$$
\langle f_2f_1,e_2e_1\rangle = (q^2-q^{-2})^{-1}(q - q^{-1})^{-1}.
$$
Furthermore, 
\begin{align*}
\langle f_1f_2,e_2e_1 \rangle &= \langle f_1\otimes f_2, k_2 e_1\otimes e_2 \rangle \\
&=\langle f_1 \otimes k_1^{-1}, e_1\otimes k_2 \rangle \langle f_2,e_2 \rangle \\
&=q^{-2} (q^2-q^{-2})^{-1}(q - q^{-1})^{-1}
\end{align*}
and 
\begin{align*}
\langle f_2f_1,e_1e_2 \rangle &= \langle f_2\otimes f_1, k_1 e_2 \otimes e_1 \rangle \\
&=\langle f_2 \otimes k_2^{-1}, e_2\otimes k_1 \rangle \langle f_1,e_1 \rangle \\
&=q^{-2} (q^2-q^{-2})^{-1}(q - q^{-1})^{-1}
\end{align*}
This proves  lines three and four.

Now for the remainder of the lemma. We have that 

%The relevant terms are
%$$
%\Delta(e_1e_2e_1) = e_1e_2 \otimes k_1 k_2 e_1 + e_2e_1 \otimes e_1k_2k_1
%$$

\begin{align*}
\langle (e_1e_2)^*f_1 , e_1e_2e_1 \rangle &= \langle  (e_1e_2)^*  \otimes f_1, e_1e_2 k_1 \otimes e_1\rangle\\
&= \langle \Delta((e_1e_2)^*)  , k_1 \otimes  e_1e_2 \rangle  \langle e_1, f_1\rangle \\
&= \langle k_1k_2 \otimes (e_1e_2)^*  , k_1 \otimes  e_1e_2 \rangle  \langle e_1, f_1\rangle \\
&= - (q-q^{-1})^{-1}
\end{align*}
and also 
\begin{align*}
\langle f_1 (e_1e_2)^* , e_1e_2e_1 \rangle &= \langle f_1 \otimes (e_1e_2)^*,  k_1k_2 e_1 \otimes e_1e_2\rangle\\
&= \langle f_1,k_1k_2e_1\rangle  \\
&=  \langle f_1,e_1\rangle \langle  k_1^{-1}, k_1k_2 \rangle\\
&= - (q-q^{-1})^{-1}
\end{align*}
and additionally 
\begin{align*}
\langle  (e_1e_2)^*f_1 ,  e_2 e_1 e_1\rangle &= 0
\end{align*}
because $e_1e_2$ is not possible as a left tensor factor. Continuing, 
\begin{align*}
\langle (e_1e_2)^*f_1,  e_1e_1e_2 \rangle &= \langle (e_1e_2)^* \otimes f_1 , e_1 k_1 e_2 \otimes  e_1   +  k_1 e_1 e_2 \otimes  e_1 \rangle \\
&= (1+q^{-2}) \langle (e_1e_2)^* \otimes f_1 ,  k_1 e_1 e_2 \otimes  e_1  \rangle \\
&= - (1+q^{-2})(q-q^{-1})^{-1} 
\end{align*}
and 
\begin{align*}
\langle f_{1}(e_1e_2)^* , e_1e_1e_2 \rangle &= \langle  f_1 \otimes (e_1e_2)^*,  e_1 k_1 k_2 \otimes e_1e_2   +   k_1 e_1 k_2 \otimes e_1e_2  \rangle\\
&= (1+q^{2})\langle  f_1 \otimes (e_1e_2)^*,  e_1 k_1 k_2 \otimes e_1e_2\rangle\\
&= - (1+q^{2})(q-q^{-1})^{-1}
\end{align*}
So we see that
\begin{align*}
\langle  (e_1e_2)^*f_1 - f_1 (e_1e_2)^*, e_1e_2e_1 \rangle &= 0 \\
\langle  (e_1e_2)^*f_1 -  f_1 (e_1e_2)^* , e_1e_1e_2 \rangle &= (q^2-q^{-2})(q-q^{-1})^{-1} = q+q^{-1}\\
\langle (e_1e_2)^*f_1 -  f_1 (e_1e_2)^* , e_2e_1e_1 \rangle &= 0
\end{align*}
and that
\begin{align*}
\langle  (e_1e_2)^*f_1 - q^{-2} f_1 (e_1e_2)^* , e_1e_2e_1 \rangle &= (-1+q^{-2})/(q-q^{-1})=-q^{-1} \\
\langle  (e_1e_2)^*f_1 -  q^{-2} f_1 (e_1e_2)^* , e_1e_1e_2\rangle &= 0 \\
\langle  (e_1e_2)^*f_1 -  q^{-2}f_1 (e_1e_2)^* , e_2e_1e_1 \rangle &= 0
\end{align*}
Furthermore, 
\begin{align*}
\langle (e_2e_1)^*f_1, e_1e_2e_1 \rangle &= \langle (e_2e_1)^* \otimes f_1, k_1e_2e_1\otimes e_1\rangle \\
&= -(q-q^{-1})^{-1} 
\end{align*}
and that
\begin{align*}
\langle f_{1}(e_2e_1)^* , e_2e_1e_1 \rangle &= \langle  f_1 \otimes (e_2e_1)^*,  k_2k_1e_1\otimes e_2e_1 + k_2e_1k_1 \otimes e_2  e_1 \rangle\\
&= (1+q^{-2})\langle    f_1 \otimes (e_2e_1)^* , k_2k_1e_1\otimes e_2e_1 \rangle\\
&= - (1+q^{-2})(q-q^{-1})^{-1}
\end{align*}
and
\begin{align*}
\langle (e_2e_1)^* f_1 , e_2e_1e_1 \rangle &= \langle  (e_2e_1)^* \otimes f_1,   e_2 e_1 \otimes k_2k_1 e_1  + e_2 e_1\otimes k_2 e_1 k_1 \rangle\\
&= (1+q^{2})\langle  (e_2e_1)^* \otimes f_1 , e_2 e_1 \otimes k_2 k_1e_1  \rangle\\
&= -(1+q^{2})(q-q^{-1})^{-1}
\end{align*}
So
so that
\begin{align*}
\langle  (e_2e_1)^*f_1 -   f_1 (e_2e_1)^* , e_1e_2e_1 \rangle &= 0 \\
\langle (e_2e_1)^*f_1 -   f_1 (e_2e_1)^* , e_1e_1e_2 \rangle &= 0 \\
\langle  (e_2e_1)^*f_1 -   f_1 (e_2e_1)^* , e_2e_1e_1 \rangle &= (q^{-2} - q^{2})(q-q^{-1})^{-1} = - (q + q^{-1})
\end{align*}
So that
\begin{align*}
(e_1e_2e_1)^* &= (q-q^{-1})(q^{-1}f_1  (e_2e_1)^* - q (e_2e_1)^*f_1) \\
&= (q-q^{-1}) \left(q^{-1} f_1(q^2f_1f_2 - f_2f_1)-q^{}(q^2f_1f_2 - f_2f_1)f_1\right)\\
&= (q-q^{-1})(q f_1f_1f_2 - (q^{-1} + q^3)f_1f_2f_1 + qf_2f_1f_1) \\
(e_2e_1e_1)^* &= \frac{q-q^{-1}}{q+q^{-1}} \left( f_1(q^2 f_1f_2 - f_2f_1) - (q^2 f_1f_2 - f_2f_1)f_1\right)\\
(e_1e_1e_2)^* &= \frac{q-q^{-1}}{q+q^{-1}} \left( (q^2 f_1f_2 - f_2f_1)f_1  - f_1(q^2 f_1f_2 - f_2f_1) \right)\\
\end{align*}
which is lines five through seven. Lines eight through ten follow from \eqref{anti}.
\end{proof}

\begin{proposition}\label{sp4central}  If $q$ is not a root of unity, the element 
\begin{multline*}
q^{-4} k_{(-2,0)} + q^{-2} k_{(0,-2)} + q^2 k_{(0,2)} + q^4 k_{(2,0)} \\
+ q^{-3} (q-q^{-1})^2 f_1 k_{(-1,-1)}e_1 + q^{-3} (e_1e_2)^* k_{(-1,1)} (f_2f_1)^* + (q^2+q^{-2})f_2e_2 \\
+ q^{-2} (e_1e_2e_1)^* (f_1f_2f_1)^* + q^{-1} (e_2e_1)^* k_{(1,-1)} (f_1f_2)^* + q^3 (q-q^{-1})^2  f_1 k_{(1,1)} e_1 
\end{multline*}
\begin{align*}
&=q^{-4} k_{(-2,0)} + q^{-2} k_{(0,-2)} + q^4 k_{(2,0)} + q^2 k_{(0,2)} \\
&+ (q-q^{-1})^2 q^{-3} f_1 k_{(-1,-1)}e_1 + (q-q^{-1})^2 q^3 f_1 k_{(1,1)}e_1 + (q^2-q^{-2})^2 f_2e_2\\
& + (q-q^{-1})^2( q^{-1} (qf_{1}f_{2}-q^{-1}f_2f_1) k_{(-1,1)}( q e_2e_1  - q^{-1} e_1e_2) + q (qf_2f_1-q^{-1}f_1f_2) k_{(1,-1)} (q e_1e_2 -  q^{-1}e_2e_1))\\
&+ (q-q^{-1})^2  \left( f_1f_1f_2 - (q^{-2} + q^2)f_1f_2f_1 +  f_2f_1f_1\right) \left( e_1e_1e_2 - (q^{-2} + q^2)e_1e_2e_1 +  e_2e_1e_1\right) 
\end{align*}
is central in $\mathcal{U}_q(\mathfrak{sp}_4)$.
\end{proposition}
\begin{proof}
Use \eqref{central} and \eqref{rho}. The terms with $\mu=\lambda$ yield
$$
q^{-4} k_{(-2,0)} + q^{-2} k_{(0,-2)} + q^2 k_{(0,2)} + q^4 k_{(2,0)}
$$
Furthermore, $\mu=\mathbf{1} > \mathbf{2} = \lambda$ yields
$$
q^{-3} (q-q^{-1})^2 f_1 k_{(-1,-1)}e_1
$$
and $\mu=\mathbf{1},\mathbf{2} > \mathbf{\bar{2}}=\lambda$ yields
$$
q^{-3} (e_1e_2)^* k_{(-1,1)} (f_2f_1)^* + (q^2+q^{-2})f_2e_2
$$
and $\mu=\mathbf{1},\mathbf{2}, \mathbf{\bar{2}} > \mathbf{\bar{1}}=\lambda$ yields
$$
q^{-2} (e_1e_2e_1)^* (f_1f_2f_1)^* + q^{-1} (e_2e_1)^* k_{(1,-1)} (f_1f_2)^* + q^3 (q-q^{-1})^2  f_1 k_{(1,1)} e_1 
$$
\end{proof}

One can check (with some calculation) that this element acts as $q^{-6} + q^{-2} + q^2 + q^6$ times the identity on $V$, which is consistent with the Harish--Chandra isomorphism.

\subsection{$\mathfrak{gl}_3$}
Recall that $\mathfrak{sl}_n$ is the rank $n-1$ Lie algebra consisting of all traceless $n\times n$ matrices. Set
$$
e_i= E_{i,i+1}, \quad f_i =E_{i+1,i}, \quad h_i= E_{ii}-E_{i+1,i+1}
$$
The simple roots and fundamental weights are
\begin{align*}
\alpha_i &= \epsilon_i - \epsilon_{i+1}, \quad  1\leq i \leq n-1, \\
 \omega_i &= \epsilon_1 + \ldots + \epsilon_n,  \quad 1\leq i\leq n,
\end{align*}
where $\epsilon_i(M)=M_{ii}$. When $n=3$ the Cartan matrix is
$$
\left(
\begin{tabular}{cc}
2 & $-1$ \\
$-1$ & 2 
\end{tabular}
\right)
$$
and the Dynkin diagram is $A_2$. Denote $k_1,k_2 \in \mathcal{U}_q(\mathfrak{sl}_3)$ by $k_{(1,-1,0)}$ and $k_{(0,1,-1)}$ respectively.

The Lie algebra $\mathfrak{gl}_3$ is the central extension of $\mathfrak{sl}_3$ by the $3\times 3$ identiy matrix. In terms of quantum groups, this corresponds to a central extension by the element $k_{(1,1,1)}$

It was shown in \cite{GZB} that the following element is central: 
\begin{multline}\label{GZBC}
C:=(q-q^{-1})^{-2}q^{-2} \Big( -(q^{-2} + 1 + q^6) + q^{-2}k_{(2,0,0)}  + k_{(0,2,0)} + q^2 k_{(0,0,2)} \\
+ (q-q^{-1})^2 (q^{-1} k_{(1,1,0)} e_1f_1 + q k_{(0,1,1)}e_2f_2 + q k_{(1,0,1)} (e_1e_2 - q^{-1}e_2e_1)(f_2f_1 - q^{-1}f_1f_2))    \Big)
\end{multline}

\section{Type $C_2$ ASEP}\label{C2}
\subsection{Notation}
Because different authors use slightly different notation, it is necessary to first establish notation for this paper. The highest weight vector of the fundamental representation $V$ is denoted $v_1$, and the lowest weight vector is denoted $v_3$ (it is essentially a coincidence that $v_3$ is the lowest weight vector in both the $A_2,C_2$ cases). The lowest weight vector corresponds to an empty site, and the highest weight vector corresponds to a completely full site. In the $A_2$ case, this means that $v_3$ is an empty site, $v_2$ is a particle of type $1$ and $v_1$ is a particle of type $2$. In the $C_2$ case, this means that $v_3$ is an empty site, $v_4$ is a particle of type $1$, $v_2$ is a particle of type $2$ and $v_1$ is a site occupied by both a particle of type $1$ and $2$. The vacuum vector $\Omega = v_3^{\otimes L}$ corresponds to $L$ lattice sites all completely empty.

There are two creation operators $e_1,e_2$. In the $A_2$ case, the operator $e_2$ creates a particle of type $2$ and the operator $e_1$ replaces a particle of type $2$ with a particle of type $1$. The annihilation operator $f_1$ replaces a particle of type $1$ with a particle of type $2$, and $f_2$ annihilates a particle of type $2$. In the $C_2$, case the operator $e_1$ creates a particle of type $1$ and $e_2$ replaces a particle of type $1$ with a particle of type $2$, and similarly for $f_1,f_2$. In a sense, $e_1,f_1$ are more accurately called ``replacement'' operators instead of creation and annihilation operators. In the $C_2$ case, $v_1$ corresponds to a site with both a type $1$ and a type $2$ particle.

Under this identification, the generator $\mathcal{L}$ of a Markov process $X_t$ on the state space $\{0,1,2\}^L$ can be identified as a linear operator on $V^{\otimes L}$. An initial condition can be expressed as a vector $A_0 \in V^{\otimes L}$ by
$$
A_0 := \sum_{v} \mathbb{P}(X_0 = v)
$$
Here, and below, the summation $\sum_v$ is over pure tensors of the form $v_{i_1} \otimes \cdots \otimes v_{i_L}$. A random variable $\mathcal{O}$ on $\{0,1,2\}^L$ can be identified with a diagonal operator on $V^{\otimes L}$ via $v\mapsto \mathcal{O}(v)v$. The same letter $\mathcal{O}$ will refer to both the random variable and the operator.

The inner product $\langle \cdot,\cdot \rangle$ on $V^{\otimes L}$ is defined by 
$$
\langle v_{i_1}\otimes \cdots \otimes v_{i_L}, v_{j_1} \otimes \cdots \otimes v_{j_L}\rangle = \delta_{i_1=j_1, \ldots, i_L=j_L}
$$
This is essentially the usual bra--ket notation. The expectation of a random variable $\mathcal{O}$ at time $t$ of a Markov process with generator $\mathcal{L}$  and initial condition $A_0$ can be computed as 
$$
\sum_{w} \langle w, \mathcal{O}e^{t\mathcal{L}}A_0\rangle.
$$ 

\subsection{Construction}
Let $C$ be the central element in Proposition \ref{sp4central} and let $A$ be the operator on $V\otimes V$ defined by 
$$
q^{-2}(q^2+q^{-2})^{-2}(q-q^{-1})^{-2}\Delta\left(C - (q^{-8} + q^{-2} + q^2 + q^8)\right).
$$
Note that $V\otimes V$ has the decomposition into nine different weight spaces (where $W(a,b)$ refers the $a \epsilon_1+b\epsilon_2$ weight space of the representation $W$)
\begin{align*}
V\otimes V &=  (V\otimes V) [2,0] \oplus (V\otimes V) [1,1] \oplus  (V\otimes V) [0,2] \oplus  (V\otimes V)[1,-1] \oplus  (V\otimes V)[0,0]\\
&= (V\otimes V)[-1,1] \oplus (V\otimes V)[0,-2] \oplus (V\otimes V)[-1,-1] \oplus (V\otimes V)[-2,0]
\end{align*}
which have dimensions $1,2,1,2, 4,2,1,2,1$ respectively. Order the basis elements of $V\otimes V$ as 
\begin{multline}\label{Order}
v_1\otimes v_1, v_2\otimes v_1,v_1\otimes v_2, v_2\otimes v_2, v_4\otimes v_1, v_1\otimes v_4, v_2\otimes v_4,v_4\otimes v_2, v_3\otimes v_1, v_1\otimes v_3, \\
 v_3\otimes v_2,v_2\otimes v_3, v_4\otimes v_4,v_3\otimes v_4,v_4\otimes v_3, v_3\otimes v_3.
\end{multline}
This ordering preserves the ordering of the weight spaces. 
 
As explained in Section \ref{Construction}, the operator $A$ needs to be conjugated with a diagonal operator corresponding to an eigenvector of $A$ with eigenvalue $0$.
\begin{lemma}\label{eigenvectors} The following are linearly independent eigenvectors of $A$ with eigenvalue $0$:
\begin{align*}
v_3 \otimes v_3 & \in (V\otimes V)[2,0]\\
 e_1(v_3 \otimes v_3) &\in (V\otimes V)[1,1] \quad\\
 e_1^2(v_3 \otimes v_3) &\in (V\otimes V)[0,2] \\
e_2e_1(v_3 \otimes v_3)&\in (V\otimes V)[1,-1]\\
e_2e_1^2(v_3 \otimes v_3)&\in (V\otimes V)[0,0]\\
e_1e_2e_1(v_3 \otimes v_3)&\in (V\otimes V)[0,0]\\
e_1^2e_2e_1(v_3 \otimes v_3)&\in (V\otimes V)[-1,1]\\
 (e_2e_1)^2(v_3 \otimes v_3)&\in (V\otimes V)[0,-2] \\
 e_1(e_2e_1)^2(v_3 \otimes v_3)&\in (V\otimes V)[-1,-1]\\
 e_1^2(e_2e_1)^2(v_3 \otimes v_3)&\in (V\otimes V)[-2,0]\\
\end{align*}
So the $0$--eigenspace of $A$ is at least $10$--dimensional.
\end{lemma}
\begin{proof}
This follows because $C(v_3 \otimes v_3) = (q^{-8} + q^{-2} + q^2 + q^8) v_3 \otimes v_3$ and $A$ commutes with $\mathcal{U}_q(\mathfrak{sp}_4)$. Note that $e_2e_1^2$ and $e_1e_2e_1$ produce linearly independent eigenvectors because the latter has $v_1\otimes v_3,v_3\otimes v_1$ terms and the former does not.
\end{proof}

Because $C\in \mathcal{U}_q(\mathfrak{sp}_4)[0],$ it follows that $A$ must preserve each summand in the weight space decomposition, so $A$ decomposes into a block matrix with $9$ blocks. By Lemma \ref{eigenvectors}, for the $1$--dimensional weight spaces with weights $(2,0),(0,2),(-2,0),(0,-2)$, the corresponding block matrices are $1\times 1$ zero matrices. Therefore $A$ has five non--zero blocks corresponding to $(1,1),(1,-1),(0,0),(-1,1),(-1,-1)$, with sizes $2,2,4,2,2$ respectively. Write this decomposition as 
$$
A =q^{-2}(q^2+q^{-2})^{-2}\left( A_{(1,1)}+ A_{(1,-1)}+ A_{(0,0)}+ A_{(-1,1)} + A_{(-1,-1)}\right)
$$
\begin{lemma}\label{symmetry}
As matrices with respect to the ordered basis in \eqref{Order},
\begin{align*}
A_{(0,0)} &= 
\left( 
\begin{tabular}{cccc}
$-q^2(q^2+q^{-2})^2$ & $ (q^2+q^{-2})^2$ & $-q^{-3}+q^{-1}+2q^3$ & $-2q^{-1} - q^3 +q^5$ \\
$(q^2+q^{-2})^2$ & $-q^{-2}(q^2+q^{-2})^2$ & $q^{-5}-q^{-3}-2q$ & $2q^{-3}+q-q^3$ \\
$-q^{-3}+q^{-1}+2q^3$ & $q^{-5}-q^{-3}-2q$ & $-q^{-4}+q^{-2}-1-2q^4-q^6$ & $(q^2+q^{-2})^2$ \\
$-2q^{-1}-q^3+q^5$ & $2q^{-3}+q-q^3$ & $(q^2+q^{-2})^2$ & $-q^{-6}-2q^{-4}-1+q^2-q^4$ 
\end{tabular}
\right)  \\
A_{(1,1)} &= A_{(1,-1)}=A_{(-1,1)}=A_{(-1,-1)}=
\left(
\begin{tabular}{cc}
$-(q^{-4}  + q^6 )  $  & $(q^{-5}+q^5)$ \\
$(q^{-5}+q^5)$ & $ -(q^{-6} + q^4) $
\end{tabular}
\right)
\end{align*}
\end{lemma}
\begin{proof}
By the definition of the co--product, the matrices for the generators can be written explicitly. For $1\leq i,j\leq 16,$  let $E_{ij}$ denote the matrix with a $1$ in the $(i,j)$--entry and $0$ elsewhere. Then  
\begin{align*}
e_1&=E_{12}+qE_{13} + q^{-1}E_{24} + E_{34} + qE_{67} + E_{68} + E_{59} + qE_{5,10} + q^{-1}E_{9,11}\\
& \quad + E_{10,11}+E_{7,12}+q^{-1}E_{8,12}+qE_{13,14}+E_{13,15}+E_{14,16}+q^{-1}E_{15,16}\\
f_1 &= q^{-1}E_{21} + E_{31} + E_{42} + qE_{43} +q^{-1}E_{95}+E_{10,5}+E_{76}+q^{-1}E_{86}+qE_{12,7}\\
& \quad +E_{12,8}+E_{11,9}+qE_{11,10}+E_{14,13}+q^{-1}E_{15,13}+qE_{16,14}+E_{16,15}\\
e_2&=E_{2,5}+E_{3,6}+q^2E_{4,8} + E_{4,10} + E_{8,13} + q^{-2} E_{10,13} + E_{12,14}+E_{11,15}\\
f_2 &= E_{5,2} + E_{6,3} + E_{8,4} + q^{-2} E_{10,4} + q^2 E_{13,8} + E_{13,10} +E_{15,11}+E_{14,12}\\
k_{(a,b)} &= \mathrm{diag}\left( q^{2a},q^{a+b},q^{a+b},q^{2b}, q^{a-b}, q^{a-b}, 1, 1, 1, 1, q^{b-a},q^{b-a},q^{-2b},q^{-a-b}, q^{-a-b}, q^{-2a}\right)
\end{align*}
Using Proposition \ref{sp4central} and explicit multiplication of $16\times 16$ matrices yields the result.

\end{proof}

Define the operator $A^{(L)}$  on $V^{\otimes L}$ by 
\begin{align*}
A^{(L)} &= \sum_{i=1}^{L-1} \mathbf{1}^{\otimes i-1} \otimes A \otimes \mathbf{1}^{\otimes L-1-i} \\
\end{align*}

\begin{lemma}\label{commutes} For any $u\in \mathcal{U}_q(\mathfrak{sp}_4)$, 
$$
[A^{(L)}, \Delta^{(L)}(u)]=0.
$$
\end{lemma}
\begin{proof}
It suffices to prove this for $u=e_i,f_i,k_i$. Since $\Delta^{(L)}(k_i) = k_i^{\otimes L}$ and 
$$
[\mathbf{1}^{\otimes i-1}, k_i^{\otimes i-1}] = [A, k_i \otimes k_i] = [\mathbf{1}^{\otimes L-1-i}, k_i^{\otimes L-i-1}]=0,
$$
this shows it for $u=k_i$. Now we have that
$$
\DeltaL(e) = \sum_{j=1}^{L-1} k^{\otimes j-1} \otimes \Delta(e) \otimes \mathbf{1}^{\otimes L-1-j}
$$
and that 
\begin{align*}
&\left[k^{\otimes j-1} \otimes \Delta(e) \otimes \mathbf{1}^{\otimes L-1-j}, \sum_{i=1}^{L-1} \mathbf{1}^{\otimes i-1} \otimes A \otimes \mathbf{1}^{\otimes L-1-i} \right] \\
&=\left[k^{\otimes j-1} \otimes \Delta(e) \otimes \mathbf{1}^{\otimes L-1-j}, \quad 
 \mathbf{1}^{\otimes j-2} \otimes A \otimes \mathbf{1}^{\otimes L-1-(j-1)} +  \mathbf{1}^{\otimes j} \otimes A \otimes \mathbf{1}^{\otimes L-1-(j+1)}\right]
\end{align*}
because for all other $j$ terms we can apply 
$$
[1\otimes 1,k\otimes k]=[\Delta(e), 1\otimes 1]=[\Delta(e),A]=[k\otimes k,A]=0.
$$
This then equals
\begin{align*}
&\left[k^{\otimes j} \otimes e \otimes \mathbf{1}^{\otimes L-1-j} + k^{\otimes j-1} \otimes e \otimes \mathbf{1}^{\otimes L-j}, \quad 
 \mathbf{1}^{\otimes j-2} \otimes A \otimes \mathbf{1}^{\otimes L-1-(j-1)} +  \mathbf{1}^{\otimes j} \otimes A \otimes \mathbf{1}^{\otimes L-1-(j+1)}\right]\\
 &=k^{\otimes j} \otimes [e\otimes 1,A] \otimes \mathbf{1}^{L-2-j} + k^{\otimes j-2} \otimes [k\otimes e,A] \otimes \mathbf{1}^{\otimes L-j}
 \end{align*}
Summing over $j$ yields
$$
\sum_{j=0}^{L-2} k^{\otimes j} \otimes [e\otimes 1,A] \otimes \mathbf{1}^{L-2-j} + \sum_{j=2}^{L} k^{\otimes j-2} \otimes [k\otimes e,A] \otimes \mathbf{1}^{\otimes L-j} = \sum_{j=0}^{L-2} k^{\otimes j} \otimes [\Delta(e),A] \otimes \mathbf{1}^{L-2-j} =0.
$$
The argument for $f$ is similar.
\end{proof}

In Lemma \ref{symmetry}, there is no value of $q$ for which the off-diagonal entries of $A_{(0,0)}$ are all non--negative, since the second row is $-q^2$ times the first row. This would indicate a  ``negative probability'' of transitioning to a state with both a type 1 and a type 2 particle occupying a site. To get around this issue, we conjugate with a  $G_{\epsilon}$ such that as $\epsilon\rightarrow 0$, these ``negative probabilities'' converge to $0$.

Give $V^{\otimes L}$ the standard basis $\mathcal{B}:=\{v_{i_1}\otimes \cdots \otimes v_{i_L}:  i_1,\ldots,i_L \in \{1,2,3,4\}\}$.  Partition  $\mathcal{B}$ into $\mathcal{B}_1 \cup \mathcal{B}_2$,  where 
$$
\mathcal{B}_1 := \{ v_{i_1} \otimes \cdots \otimes v_{i_L}:  i_1,\ldots, i_L \in \{2,3,4\}\}
$$ 
Note that  $\left| \mathcal{B}_1 \right| = 3^L$ and $\left| \mathcal{B}_2 \right| = 4^L - 3^L$. 

Define the sets
\begin{align*}
\mathcal{E}_1 &= \{ e_2^j e_1^k: 1\leq j\leq k\leq L\} \\
\mathcal{E}_2 &= \{ e_1^i e_2^j e_1^k: 1\leq i\leq j\leq k\leq L\}
\end{align*}
Let $\Omega$ be the vacuum vector $v_3^{\otimes L}$. We then have that
$$
e(\Omega) \in \text{span}(\mathcal{B}_1) \text{ for all } f \in \mathcal{E}_1, \quad e(\Omega) \notin \text{span}(\mathcal{B}_1) \text{ for all } f\in \mathcal{E}_2 
$$

Let $g_{\epsilon} \in V^{\otimes L}$ be a vector in the kernel of $A^{(L)}$, and for $x \in \mathcal{B},$ define $g_{\epsilon}(x)$ to be the coefficient of $x$ in $g_{\epsilon}.$ Suppose it satisfies
\begin{equation}\label{niceg}
g_{\epsilon}(x)>0 \text{ for all } x\in \mathcal{B}, \quad \lim_{\epsilon \rightarrow 0} g_{\epsilon}(y) = 0 \text{ for } y\in \mathcal{B}_2, \quad \lim_{\epsilon \rightarrow 0} g_{\epsilon}(x) > 0 \text{ for } x\in \mathcal{B}_1, 
\end{equation}
Let $G_{\epsilon}$ be the diagonal matrix on $V^{\otimes L}$ with entries $G_{\epsilon}(x,x)=g_{\epsilon}(x)$. Let $\mathcal{L}_{\epsilon}$ be 
\begin{equation}\label{L}
\mathcal{L}_{\epsilon} = G_{\epsilon}^{-1} A^{(L)} G_{\epsilon}.
\end{equation}
For a matrix $S$ that commutes with $A^{(L)}$, let $D_{\epsilon}=G_{\epsilon}^{-1}SG_{\epsilon}^{-1}$. In the $\epsilon\rightarrow 0$ limit, the subscript will be dropped. The idea for this construction of $D$ comes from Proposition 2.1 of \cite{CGRS}.

\begin{proposition}\label{duality}
If $x \in \mathrm{span}({\mathcal{B}}_1)$, then
$$
\lim_{\epsilon\rightarrow 0} \langle y, \mathcal{L}_{\epsilon}D_{\epsilon}(x)\rangle = \lim_{\epsilon\rightarrow 0} \langle y, D_{\epsilon}\mathcal{L}^*_{\epsilon}(x)\rangle \text{ for all } y \in \mathrm{span}({\mathcal{B}}_1)
$$
(and this limit is finite).

\end{proposition}
\begin{proof}
Since $A^{(L)}$ is symmetric,
$$
\mathcal{L}_{\epsilon}D_{\epsilon}  = G_{\epsilon}^{-1}A^{(L)}SG_{\epsilon}^{-1} = G_{\epsilon}^{-1}S G_{\epsilon}^{-1} G_{\epsilon} A^{(L)}G_{\epsilon}^{-1} = D_{\epsilon} \mathcal{L}_{\epsilon}^* 
$$
so it remains to check that the limit is finite. But by \eqref{niceg}, the limit can only be infinite if $x$ or $y$ is not in the span of $\mathcal{B}_1$.
\end{proof}

In order to find an explicit $g_{\epsilon}$ satisfying \eqref{niceg}, introduce some notation first.

\begin{definition}
The $q$--analog of the exponential function is
$$
\mathrm{exp}_q(x) := \sum_{n=1}^{\infty} \frac{x^n}{\{n\}_q!}
$$
where 
$$
\{n\}_q :=\frac{1-q^n}{1-q}.
$$
\end{definition}
The following is Proposition 5.1 from \cite{CGRS}. 
\begin{proposition}\label{pseudofac}
Let $\{g_i,k_i:1\leq i\leq L\}$ be operators such that $k_ig_i=rg_ik_i$. Define
$$
k^{(i)}:=k_1\cdots k_i, \quad g^{(L)}:= \sum_{i=1}^L k^{(i-1)}g_i, \quad h^{(i)}:=k_i^{-1}\cdots k_{L}^{-1}, \quad \hat{g}^{(L)}:= \sum_{i=1}^L g_i h^{(i+1)}.
$$
Then
\begin{align*}
\exp_r(g^{(L)}) &= \exp_r(g_1)\cdot \exp_r(k^{(1)}g_2) \cdot \cdots \cdot \exp_r(k^{(L-1)}g_L)   \\
\exp_r(\hat{g}^{(L)}) &= \exp_r(g_1 h^{(2)} )\cdot \cdots \cdot \exp_r(g_{L-1}h^{(L)})  \exp_r(g_L) 
\end{align*}
\end{proposition}
In this paper, the proposition will be applied with
\begin{align*}
g_i &= 1^{\otimes i-1} \otimes e \otimes 1^{\otimes L-i}\\
k_i &= 1^{\otimes i-1}\otimes k\otimes 1^{\otimes L-i}
\end{align*}
where $e,k$ can be either $e_1,k_1$ or $e_2,k_2$. Note that the $L$--fold co--product $\Delta^{(L)}e$ is of the form $g^{(L)}$ in the proposition.

Now let 
\begin{equation}\label{goodg}
g_{\epsilon}:= \left(\exp_{q^4}\left({\DeltaL} e_2 \right) \cdot \exp_{q^2}\left({\DeltaL} e_1\right) + \epsilon \sum_{e \in \mathcal{E}_2} \Delta^{(L)}e \right) (v_3 \otimes \cdots \otimes v_3) 
\end{equation}
It is immediate from the definitions that \eqref{niceg} holds. The fact that $g_{\epsilon}$ is in the kernel of $A^{(L)}$ follows from Lemma \ref{commutes}. The first statement in Theorem \ref{C2Construction} can now be proved.

\begin{theorem}\label{C2gen} The restriction of  $\mathcal{L}$ to ${\mathcal{B}_1}$  is the generator of spin $1/2$ type $C_2$ ASEP on $\{1,\ldots,L\}$ with domain wall boundary conditions.
\end{theorem}
\begin{proof}
We will use the lemma:

\begin{lemma}\label{generator}
The generator of the generalized two particle type ASEP on $\{1,\ldots,L\}$ with domain wall boundary conditions is of the form
$$
\sum_{i=1}^{L-1} 1^{\otimes i-1} \otimes H \otimes 1^{L-i-1}
$$
where the matrix of $H$ with respect to the basis $(0,0),(0,1),(1,0),(1,1),(2,1),(1,2),(2,2),(0,2),(2,0)$ is
$$
\left(
\begin{array}{ccccccccc}
 0 & 0 & 0 & 0 & 0 & 0 & 0 & 0 & 0 \\
 0 & -L(1,0) & L(1,0) & 0 & 0 & 0 & 0 & 0 & 0 \\
 0 & R(1,0) & -R(1,0)  & 0 & 0 & 0 & 0 & 0 & 0 \\
 0 & 0 & 0 & 0 & 0 & 0 & 0 & 0 & 0 \\
 0 & 0 & 0 & 0 &  -L(1,2) & L(1,2) & 0 & 0 & 0 \\
 0 & 0 & 0 & 0 & R(1,2) & -R(1,2) & 0 & 0 & 0 \\
 0 & 0 & 0 & 0 & 0 & 0 & 0 & 0 & 0 \\
 0 & 0 & 0 & 0 & 0 & 0 & 0 & -L(2,0) & L(2,0) \\
 0 & 0 & 0 & 0 & 0 & 0 & 0 & R(2,0) & -R(2,0) \\
\end{array}
\right)
$$
\end{lemma}
\begin{proof}
Since the particles in two particle type ASEP only jump at most one site, and all the jump bond rates are the same, the generator can be written in that form. The matrix entries can be found from the definition of a generator of a Markov process.
\end{proof}
To finish the proof of the theorem, it remains to show that $G^{-1}A^{(L)}G$ matches the expression in the lemma. From Proposition \ref{pseudofac}, $G$ can be written in the form
$$
G(v)= g_1(v) \cdots g_L(v) v, \text{ for } v=v_{i_1}\otimes \cdots \otimes v_{i_L}
$$
where $g_j(v_{i_1}\otimes \cdots \otimes v_{i_L})$ only depends on the values of $i_1,\ldots,i_j$ irrespective of order. In other words, $g_j$ only depends on the cardinalities of the sets $\{k: 1\leq k\leq j, i_k=0\}, \{k: 1\leq k\leq j, i_k=1\}$. Thus, 
$$
G^{-1}A^{(L)}G(v) = \sum_{i=1}^L 1^{\otimes i-1} \otimes H\otimes 1^{L-i-1}(v)
$$
where $H=B^{-1}AB$ for some diagonal matrix $B$. 

Since $g$ is in the kernel of $A^{(L)}$, each row of $H$ must sum to $0$. Conjugating by a diagonal matrix does not change the diagonal entries, so Lemma \ref{symmetry} shows that  $H$ has the necessary form. 
\end{proof}

\subsection{Duality}

We first prove an equivalent definition of duality.
\begin{lemma}\label{DualityProof}
Suppose that $\mathcal{L}$ is the generator of the Markov process $X(t)$ on state space $X$. Let $D$ be a function on $X\times X$ viewed as an operator in the sense of the formal sum 
$$
D(y) = \sum_{x\in X} D(x,y) \mathbf{x}.
$$
If $Z,Y$ are subsets of $X$ such that for all $(z,y)\in Z\times Y$ 
$$
\langle z, \mathcal{L}D(y)\rangle  = \langle z, D\mathcal{L}^*(y)\rangle
$$
then $Z \times Y \subseteq \mathcal{S}_D$.
\end{lemma}
\begin{proof}
By definition
\begin{align*}
e^{t\mathcal{L}}D(y)  &= \sum_{x} e^{t\mathcal{L}}\left( D(x,y)\mathbf{x} \right)\\
&= \sum_{x,z} D(x,y) e^{t\mathcal{L}}(z,x)\mathbf{z}\\
&= \sum_{x,z} \mathbb{P}_t(z \rightarrow x)D(x,y) \mathbf{z}
\end{align*}
and 
\begin{align*}
De^{t\mathcal{L}^*}(y) &= \sum_{x} D \left( e^{t\mathcal{L}}(y,x)\mathbf{x} \right)\\
&= \sum_{z,x} D(z,x) e^{t\mathcal{L}}(y,x)\mathbf{z}\\
&= \sum_{z,x} \mathbb{P}_t(y\rightarrow x)D(z,x)\mathbf{z}
\end{align*}
By the assumptions of the lemma this implies that for all $z\in Z$ and $y\in Y$,
\begin{equation}\label{trouble}
\sum_{x} \mathbb{P}_t(z \rightarrow x)D(x,y) = \sum_{x} \mathbb{P}_t(y\rightarrow x)D(z,x)
\end{equation}
which is equivalent to saying  that for all $z\in Z,y\in Y $, 
$$
\mathbb{E}_z[D(X(t),y)] = \mathbb{E}_y[D(z,X(t))],
$$
which means exactly that $(z,y)\in \mathcal{S}_D$.
\end{proof}

By Proposition \ref{duality}, $D$ can be used to obtain a suitable duality function. The difficulty lies in the simple fact: in the equation \eqref{trouble}, ignoring the summation over states $x$ with sites containing both a particle of type $1$ and a particle of type $2$ will not always leave the sum unchanged. However, certain duality functions will still work:

\begin{lemma}\label{Proper} Suppose $y,z$ are such that
$$
D(x,y)=D(z,x)=0 \text{ for all } x\notin \mathrm{span}(\mathcal{B}_1)
$$
Then $(z,y)\in \mathcal{S}_D$.
\end{lemma}
\begin{proof}
With the assumptions of the lemma, the summation over $x\in \mathrm{span}(\mathcal{B}_1)$ in \eqref{trouble} is $0$, as needed.
\end{proof}
 
Now it remains to find proper duality functions $D$ satisfying Lemma \ref{Proper}. There are two natural choices. The first is to consider 
$$
 S:=\exp_{q^4}\left({\DeltaL} e_2 \right) \cdot \exp_{q^2}\left({\DeltaL} e_1\right) 
$$
and set $D_{\epsilon}=G_{\epsilon}^{-1}SG_{\epsilon}^{-1}$, with $D=\lim_{\epsilon\rightarrow 0}D_{\epsilon}$. The idea behind this choice is as follows. In order for Lemma \ref{Proper} to hold, the symmetry $S$ should not create a site with both a type $1$ and a type $2$ particle. Since $e_1$ creates a particle of type $1$ and $e_2$ replaces a particle of type $1$ with a particle of type $2$, this holds as long as $\xi$ does not contain any particles of type $2$.

Below, recall that
$$
v_3 \in V(-1,0), \quad v_4 \in V(0,-1), \quad v_2 \in V(0,1)
$$

\begin{proposition}\label{computes} If $\xi_i=0,1$ for all $i$, then 
$$
S(\eta,\xi) = \prod_{i=1}^L 1_{\{\xi_i\leq \eta_i\}} q^{1_{\{\xi_i=0,\eta_i\neq 0\}} \sum_{j=1}^{i-1} {\color{black}(1_{\{\xi_j=1\}} - 1_{\{\xi_j=0\}} ) }}  (q^{-2})^{1_{\{\eta_i=2\}} \sum_{j=1}^{i-1} (1_{\{\xi_j=0,\eta_j\neq 0\}}+ 1_{\{\xi_j=1\}} ) } 
$$
\end{proposition}
\begin{proof} Use Proposition \ref{pseudofac}. Since $e_1^2$ and $e_2^2$ act as $0$ on $V$, it is equivalent to consider
\begin{multline*}
(1+e_2\otimes 1^{L-1})(1+k_2\otimes e_2\otimes 1^{L-2})\ldots (1+(k_2)^{\otimes (L-1)}\otimes e_2)\\
(1+e_1\otimes 1^{L-1})(1+k_1\otimes e_1\otimes 1^{L-2})\ldots (1+(k_1)^{\otimes (L-1)}\otimes e_1).
\end{multline*}
First, move the $e_2$ terms from left to right to get 
\begin{multline*}
(1+e_2\otimes 1^{L-1})(1+e_1\otimes 1^{L-1}) (1+k_2\otimes e_2\otimes 1^{L-2})(1+k_1\otimes e_1\otimes 1^{L-2}) \\
 \ldots (1+(k_2)^{\otimes (L-1)}\otimes e_2) (1+(k_1)^{\otimes (L-1)}\otimes e_1).
\end{multline*}
Due to the commutation relation $k_2e_1 = q^{-2}e_1k_2$, this produces the term
$$
\prod_{i=1}^L (q^{-2})^{1_{\{\eta_i=2\}} \sum_{j=1}^{i-1} 1_{\{\xi_j=0,\eta_j\neq 0\}}}
$$
Next, applications of the $e_1$ terms to $\xi$ yield 
$$
\prod_{i=1}^L q^{1_{\{\xi_i=0,\eta_i\neq 0\}} \sum_{j=1}^{i-1}  {\color{black} (1_{\{\xi_j=1\}} - 1_{\{\xi_j=0\}}} )}   .
$$
And then the applications of the $e_2$ yields
$$
\prod_{i=1}^L (q^{-2})^{1_{\{\eta_i=2\}}  \sum_{j=1}^{i-1}   1_{\{\xi_j=1\}}}
$$
and combining all three lines gives the result.
\end{proof}

Recall  
\begin{align*}
N_k^R(\eta) &= \left| \{j > i: \eta_j \neq 0\} \right|\\
N_k^L(\eta) &= \left| \{i < j: \eta_j \neq 0\} \right|.
\end{align*}
For $n_1<\ldots<n_r$, let $\xi^{(n_1,\ldots,n_r)}$ be the state where $\xi_{n_s}=1$ and all other $\xi_i=0$. As before, $\Omega$ is the vacuum vector. Proposition \ref{computes} immediately implies:

\begin{corollary}
We have 
$$
G(\eta):=S(\eta,\Omega) = \prod_{i=1}^L  q^{{\color{black} 1_{\{\eta_i\neq 0\}} (1-i)}}   (q^{-2})^{1_{\{\eta_i=2\}} N_i^L(\eta) } 
$$
$$
G(\xi^{(n_1,\ldots,n_r)}) = \prod_{s=1}^r q^{1-n_s}
$$
And
\begin{align*}
S(\eta,\xi^{(n_1,\ldots,n_r)}) &= \prod_{s=0}^r 1_{\{\eta_{n_s}\neq 0\}}(q^{-2})^{1_{\{\eta_s=2\}}N_{n_s}^L(\eta)}\prod_{i=n_s+1}^{n_{s+1}-1} q^{1_{\{\eta_i\neq 0\}}(2s-i+1)} (q^{-2})^{1_{\{\eta_i=2\}}N_{n_s}^L(\eta)}\\
&=1_{\{\eta_{n_1},\ldots,\eta_{n_r}\neq 0\}} \prod_{i=1}^L (q^{-2})^{1_{\{\eta_i=2\}}N_{n_s}^L(\eta)} \times \prod_{s=0}^r \prod_{i=n_s+1}^{n_{s+1}-1} q^{1_{\{\eta_i\neq 0\}}(2s-i+1)} 
\end{align*}
\end{corollary}

%$$
%S(\eta,\xi^{(n)}) = 1_{\{\eta_n>0\}} \prod_{i=1}^{n-1}  q^{1_{\{\eta_i\neq 0\}} {\color{black}(1-i)} }  (q^{-2})^{1_{\{\eta_i=2\}} N^L_i(\eta)}   \cdot  (q^{-2})^{ 1_{\{\eta_n=2\}} N^L_n(\eta)} 
%\prod_{i=n+1}^L q^{{\color{black} 1_{\{\eta_i\neq 0\}}(-i+3)  }}  (q^{-2})^{1_{\{\eta_i=2\}} (N^L_i(\eta)  )} 
%$$
Theorem \ref{C2Duality}(1) can now be proved. Suppose that $\eta_i=2$ exactly when $i\in \{m_1,\ldots,m_l\}$ (and possibly $1$ elsewhere). Then
$$
S(\eta,\xi^{(n_1,\ldots,n_r)}) = 1_{\{\eta_{n_1},\ldots,\eta_{n_r}\neq 0\}} \prod_{k=1}^l q^{-2N_{m_k}^L(\eta)} \times \prod_{s=0}^r \prod_{i=n_s+1}^{n_{s+1}-1} q^{1_{\{\eta_i\neq 0\}}(2s-i+1)} 
$$
so that
\begin{align*}
D(\eta,\xi^{(n_1,\ldots,n_r)}) &= \frac{1}{G(\eta)}1_{\{\eta_{n_1},\ldots,\eta_{n_r}\neq 0\}} \prod_{k=1}^l q^{-2N_{m_k}^L(\eta)} \times \prod_{s=0}^r q^{n_s-1}\prod_{i=n_s+1}^{n_{s+1}-1} q^{1_{\{\eta_i\neq 0\}}(2s-i+1)} \\
&=1_{\{\eta_{n_1},\ldots,\eta_{n_r}\neq 0\}} \prod_{i=1}^L q^{1_{\{\eta_i\neq 0\}}(i-1)} \times \prod_{s=0}^r q^{n_s-1}\prod_{i=n_s+1}^{n_{s+1}-1} q^{1_{\{\eta_i\neq 0\}}(2s-i+1)} \\
&=1_{\{\eta_{n_1},\ldots,\eta_{n_r}\neq 0\}} \prod_{s=0}^r q^{2(n_s-1)}\prod_{i=n_s+1}^{n_{s+1}-1} q^{1_{\{\eta_i\neq 0\}}(2s)} \\
&=1_{\{\eta_{n_1},\ldots,\eta_{n_r}\neq 0\}} \prod_{s=1}^r q^{2(n_s-1)}  q^{2(N_{n_s}^R(\eta)-(r-s))}\\
&= q^{-2r - (r-1)r}\prod_{s=1}^r 1_{\{\eta_{n_s}\neq 0\}} q^{2n_s + 2N_{n_s}^R(\eta)}
\end{align*}
which is Theorem \ref{C2Duality}(2). 

%This all implies that if only $\eta_m=2$ then
%$$
%D(\eta,\xi^{(n)}) = \frac{S(\eta,\xi^{(n)})}{G(\eta)G(\xi^{(n)})} = \frac{S(\eta,\xi^{(n)})}{G(\eta)q^{1-n}}
%=   1_{\{\eta_n>0\}}q^{2(n-1)}  \prod_{i=n+1}^L q^{{\color{black} 1_{\{\eta_i\neq 0\}}(2)  }}   
%$$

Now consider the case when 
$$
S=exp_{q^4}\left({\DeltaL} e_2 \right).
$$
In this case, any $\xi$ will work.

\begin{lemma}
$$
S(\eta,\xi) = \prod_{i=1}^L \left( 1_{\{\xi_i = \eta_i\}} +  1_{\{\xi_i=1,\eta_i=2\}}(q^2)^{\sum_{j=1}^{i-1} 1_{\{\xi_j=2\}} - 1_{\{\xi_j=1\}} }\right)
$$
\end{lemma}
\begin{proof}
The applications of the $e_2$ only occur when $\xi_i=1,\eta_i=2$, and the lemma follows because $k_{(0,2)}$ maps $v_3$ to $v_3$, $v_4$ to $q^{-2}v_4$ and $v_2$ to $q^2v_2$.
\end{proof}

Since 
$$
G(\eta) = \prod_{i=1}^L  q^{{\color{black} 1_{\{\eta_i\neq 0\}} (1-i)}}   (q^{-2})^{1_{\{\eta_i=2\}} N_i^L(\eta) } 
$$
we have
$$
D(\eta,\xi) = \prod_{i=1}^L \left( 1_{\{\xi_i = \eta_i=1\}}q^{2(i-1)} +  1_{\{\xi_i = \eta_i=2\}}q^{2(i-1 + N^L_i(\eta) + N^L_i(\xi))}  +  1_{\{\xi_i=1,\eta_i=2\}}(q^2)^{N_i^L(\eta) + i-1 + \sum_{j=1}^{i-1} (1_{\{\xi_j=2\}} - 1_{\{\xi_j=1\}} )}\right)
$$
which simplifies to the expression in Theorem \ref{C2Duality}(1).

\section{Type $A_2$ ASEP}\label{A2}
Let $C$ be the central element of $\mathcal{U}_q(\mathfrak{gl}_3)$ from \eqref{GZBC}.
\begin{lemma}\label{A2Gen}
With respect to the basis $v_1\otimes v_1, v_2\otimes v_2,v_3\otimes v_3, v_2\otimes v_1, v_1\otimes v_2, v_3\otimes v_1,v_1\otimes v_3, v_3\otimes v_2,v_2\otimes v_3,$ the matrix of $\Delta (C)$ on $V\otimes V$ is
$$
\left(
\begin{array}{ccccccccc}
 0 & 0 & 0 & 0 & 0 & 0 & 0 & 0 & 0 \\
 0 & 0 & 0 & 0 & 0 & 0 & 0 & 0 & 0 \\
 0 & 0 & 0 & 0 & 0 & 0 & 0 & 0 & 0 \\
 0 & 0 & 0 & -q^2 & q & 0 & 0 & 0 & 0 \\
 0 & 0 & 0 & q & -1 & 0 & 0 & 0 & 0 \\
 0 & 0 & 0 & 0 & 0 & -q^2 & q & 0 & 0 \\
 0 & 0 & 0 & 0 & 0 & q & -1 & 0 & 0 \\
 0 & 0 & 0 & 0 & 0 & 0 & 0 & -q^2 & q \\
 0 & 0 & 0 & 0 & 0 & 0 & 0 & q & -1 \\
\end{array}
\right)
$$
\end{lemma}
\begin{proof} By computation:
\begin{align*}
e_1 &= q^{-1}E_{42} + E_{52} + E_{14} + qE_{15} + E_{68}+E_{79}\\
f_1 &= q^{-1}E_{41} + E_{51} + E_{24}+qE_{25} + E_{86} + E_{97} \\
k_{(a,b,c)} &= \mathrm{diag}\left( q^{2a},q^{2b},q^{2c}, q^{a+b},q^{a+b},q^{a+c},q^{a+c},q^{b+c},q^{b+c}\right)\\
e_2 &=q^{-1}E_{83} + E_{93} +E_{46} + E_{57} + E_{28} + q E_{29}\\
f_2 &= q^{-1}E_{82} + E_{92} + E_{64} + E_{75} + E_{38} + qE_{39} 
\end{align*}

\end{proof}

The symmetry in this case is 
$$
 S:=\exp_{q^2}\left({\DeltaL} e_2 \right) \cdot \exp_{q^2}\left({\DeltaL} e_1\right) 
$$
\begin{proposition}\label{ExplicitS}
$$
S(\eta,\xi)=\prod_{i=1}^L q^{1_{\{\xi_i = 0, \eta_i\neq 0\}} \sum_{j=1}^{i-1} (1_{\xi_j=1} - 1_{\xi_j=0})} (q^{-1})^{1_{\{\eta_i=2,\xi_i\neq 2\}}\sum_{j=1}^{i-1} ( 1_{\{\xi_j=0, \eta_j\neq 0\}} + 1_{\{\xi_j=1\}} - 1_{\{\xi_j=2\}})}
$$
implying that
$$
G(\eta):=S(\eta,\Omega)= \prod_{i=1}^L q^{1_{\{\eta_i\neq 0\}}(1-i)}(q^{-1})^{ 1_{\{\eta_i=2\}} \sum_{j=1}^{i-1} 1_{\{\eta_j\neq 0\}} }
$$
\end{proposition}
\begin{proof}
The argument is identical to that of Proposition \ref{computes}. 
\end{proof}
Therefore we see that
\begin{theorem}
The operator
$$
\mathcal{L}:=G^{-1}A^{(L)}G
$$
is the generator of spin $1/2$ type $A_2$ ASEP on $\{1,\ldots,L\}$ with domain wall boundary conditions.

The function 
$$
D:=G^{-1}SG^{-1}
$$
is  a self--duality function explicitly given by the expression given in Theorem \ref{SDF}.
\end{theorem}
\begin{proof}
The first statement follows from an argument similar to that of Theorem \ref{C2gen}. The second statement is a direct computation using Proposition \ref{ExplicitS}.
\end{proof}

By Proposition \ref{ExplicitS}, 
\begin{align*}
D(\eta,\xi) &= \prod_{i=1}^L q^{1_{\{\xi_i = 0, \eta_i\neq 0\}} \sum_{j=1}^{i-1} (1_{\{\xi_j=1\}} - 1_{\{\xi_j=0\}})} (q^{-1})^{1_{\{\eta_i=2,\xi_i\neq 2\}}\sum_{j=1}^{i-1} ( 1_{\{\xi_j=0, \eta_j\neq 0\}} + 1_{\{\xi_j=1\}} - 1_{\{\xi_j=2\}})}  \\
&\quad \quad \times q^{1_{\{\eta_i\neq 0\}}(i-1)}q^{ 1_{\{\eta_i=2\}} \sum_{j=1}^{i-1} 1_{\{\eta_j\neq 0\}} } q^{1_{\{\xi_i\neq 0\}}(i-1)}q^{ 1_{\{\xi_i=2\}} \sum_{j=1}^{i-1} 1_{\{\xi_j\neq 0\}} }
\end{align*}
which equals $\prod_{i=1}^L f(\eta_i,\xi_i)$ where $f(\cdot,\cdot)$ equals 
\begin{align*}
1, &\text{ if } \xi_i=0, \eta_i=0\\ 
q^{\sum_{j=1}^{i-1} \left(1_{\{\xi_j=1\}} - 1_{\{\xi_j=0\}}\right)} \cdot q^{i-1}, &\text{ if } \xi_i=0,\eta_i=2 \\
q^{\sum_{j=1}^{i-1} \left(1_{\{\xi_j=1\}} - 1_{\{\xi_j=0\}}\right)} \cdot q^{-\sum_{j=1}^{i-1} ( 1_{\{\xi_j=0,\eta_j\neq 0\}} + 1_{\{\xi_j=1\}} - 1_{\{\xi_j=2\}}) } \cdot q^{i-1 + N_i^L(\eta) }, &\text{ if } \xi_i=0,\eta_i=1 \\
q^{2(i-1)}, &\text{ if } \xi_i=2,\eta_i=2\\
q^{-\sum_{j=1}^{i-1} (1_{\{\xi_j=0,\eta_j\neq 0\}} + 1_{\{\xi_j=1\}} - 1_{\{\xi_j=2\}}) + 2(i-1) + N_i^L(\eta) } , &\text{ if } \xi_i=2,\eta_i=1\\
q^{2(i-1) + N_i^L(\xi) + N_i^L(\eta)}, &\text{ if } \xi_i=1,\eta_i=1
\end{align*}
If there are $s_2$ type $2$ particles and $s$ type $1$ particles in $\xi$ to the left of $i$, then this becomes
\begin{align*}
1, &\text{ if } \xi_i=0, \eta_i=0\\ 
q^{ s_2 - (i - 1 - s_2-s   ) } \cdot q^{i-1} = q^{2s_2+s}, &\text{ if } \xi_i=0,\eta_i=2 \\
q^{ s_2 - (i - 1 - s_2-s  ) } \cdot q^{2 {r}} \cdot q^{i-1 } = q^{2s_2+3s}, &\text{ if } \xi_i=0,\eta_i=1 \\
q^{2(i-1)}, &\text{ if } \xi_i=2,\eta_i=2\\
q^{2s + 2(i-1)  } , &\text{ if } \xi_i=2,\eta_i=1\\
q^{2(i-1) + s + s_2 + N_i^L(\eta)}, &\text{ if } \xi_i=1,\eta_i=1
\end{align*}

The $q^{2s}$ term in the fifth line and $q^{s+s_2}$ in the sixth line result in a contribution from the configuration of $\xi$. If $\xi$ has a total of $r$ type $1$ particles all to the left of $r'$ type $2$ particles, then the contribution is $ q^{(r-1)r/2 +r'r}$, which is constant with respect to the dynamics of $\xi$.  Each time a type $1$ particle jumps to the right of  a type $2$ particle, the contribution is unchanged, and hence remains a constant. 

Let $\xi$ denote the particle configuration with particles of type $1$ at $n_1,\ldots,n_r$ and particles of type $2$ at $m_1,\ldots,m_{r'}$.  The sixth line yields
$$
\prod_{s=1}^{r} q^{N_{n_s}^L(\eta)} = \prod_{i=1}^L q^{1_{\{\eta_i\neq 0\}} \cdot \tilde{N}^R_i(\xi)} = \prod_{s=0}^r q^{r-s}\prod_{i = n_s + 1}^{n_{s+1}-1} q^{1_{\{\eta_i\neq 0\}} \cdot (r-s)} = \mathrm{const} \prod_{s=0}^r \prod_{i = n_s + 1}^{n_{s+1}-1} q^{1_{\{\eta_i\neq 0\}} \cdot (r-s)}
$$
This combines with the  $q^s$ and $q^{3s}$ in the second and third lines to contribute
\begin{align*}
\prod_{s=0}^r \prod_{i = n_s + 1}^{n_{s+1}-1} q^{1_{\{\eta_i\neq 0\}} \cdot (r-s)} q^{s\cdot 1_{\{\eta_i\neq 0\}}}  q^{2s\cdot 1_{\{\eta_i=1\}}} = \mathrm{const} \prod_{s=1}^r q^{2s\left( \tilde{N}^L_{n_{s+1}}(\eta) - \tilde{N}_{n_s}^L(\eta) -1  \right) }
\end{align*}
Similarly, the $2s_2$ contributes
$$
\prod_{s'=1}^{r'} \prod_{i = m_{s'} + 1}^{m_{s'+1}-1} q^{2s' \cdot 1_{\{\eta_i\neq 0\}}} = \mathrm{const} \prod_{s'=1}^{r'} q^{2s'\left( {N}^L_{m_{s'+1}}(\eta) - {N}_{m_{s'}}^L(\eta) -1  \right) }
$$
Combining the terms yields 
$$
D(\eta,\xi) = \mathrm{const}  \prod_{s=1}^r 1_{\{\eta_{n_s}=1\}}q^{2\tilde{N}^R_{n_s}(\eta)+2n_s} \prod_{s'=1}^{r'} 1_{\{\eta_{m_{s'}}\neq 0\}}q^{2N^R_{m_{s'}}(\eta)+2m_{s'}} ,
$$
which proves Theorem \ref{SDF}.

We remark that Theorem \ref{SDF} provides a formula for the $r+r'$ moments of the exponentiated current of type $A_2$ ASEP at distinct points. By following the argument in \cite{IS}, it should be possible to write the moments at any points in terms of $k$--particle evolution for $k\leq r+r'$, but this is not pursued here.

\bibliographystyle{alpha}

\end{document}